\documentclass[12pt,twoside]{amsart}
\usepackage{geometry}
\usepackage{amssymb,amsmath,amsthm, amscd, enumerate, mathrsfs}
\usepackage{graphicx, hhline}
\usepackage[all]{xy}
\usepackage[dvipdfmx]{hyperref}
\usepackage[shortlabels]{enumitem}
\setlist[enumerate]{leftmargin=56pt,labelsep=
8pt,itemsep=4pt,label=\upshape{(\thethm.\arabic*)}}
\usepackage{mathtools}
\mathtoolsset{showonlyrefs}
\usepackage{xcolor}

\title{Remarks on log pluricanonical representations}
\author{Osamu Fujino and Jinsong Xu}
\date{2025/1/16, version 0.06}
\address{Department of 
Mathematics, Graduate School of Science, 
Kyoto University, Kyoto 606-8502, Japan}
\email{fujino@math.kyoto-u.ac.jp}
\address{Department of Mathematical Sciences\\ 
Xi'an Jiaotong-Liverpool University\\ 
No.111 Ren'ai Road, SIP, Suzhou, Jiangsu Province, China}
\email{Jinsong.Xu@xjtlu.edu.cn}
\DeclareMathOperator{\Bir}{Bir}
\DeclareMathOperator{\Aut}{Aut}
\DeclareMathOperator{\PBir}{PBir}

\newtheorem{thm}{Theorem}[section]

\theoremstyle{definition}
\newtheorem{defn}[thm]{Definition}
\newtheorem{rem}[thm]{Remark}

\newtheorem*{ack}{Acknowledgments}

\makeatletter

\@addtoreset{equation}{section}
\makeatother
\begin{document}

\begin{abstract}
We show the finiteness of log pluricanonical representations under the assumption of the existence of a good minimal model.
\end{abstract}

\maketitle 

\tableofcontents 

\section{Introduction}

This short paper is a supplement to \cite{fujino-gongyo} (see also \cite{hacon-xu}). 
One of the main purposes of this paper is to establish: 

\begin{thm}\label{thm1.1}
Let $(X, \Delta)$ be a projective log canonical pair 
such that $K_X+\Delta$ is 
$\mathbb Q$-Cartier. 
Assume that $(X, \Delta)$ has a good minimal model. 
Then there exists a positive integer $k$ such that 
the image of 
\begin{equation*}
\rho_{km}\colon \Bir(X, \Delta)\to \Aut _{\mathbb C} 
H^0(X, \mathcal O_X(km(K_X+\Delta)))
\end{equation*} 
is a finite group for every positive integer $m$, 
where 
\begin{equation*}
\Bir(X, \Delta):=\{\sigma\mid \text{$\sigma\colon (X, \Delta)
\dashrightarrow (X, \Delta)$ is 
$B$-birational}\}. 
\end{equation*}
\end{thm}

We make five important remarks on Theorem \ref{thm1.1}. 

\begin{rem}\label{rem1.2}
In Theorem \ref{thm1.1}, 
$k$ is the smallest positive integer 
such that $k(K_X+\Delta)$ and $k(K_{X'}+\Delta')$ are both 
Cartier, where $(X', \Delta')$ is a good minimal model of $(X, \Delta)$. 
\end{rem}

\begin{rem}\label{rem1.3}
If $K_X+\Delta$ is semiample, that is, $(X, \Delta)$ itself is a 
good minimal model of $(X, \Delta)$, then Theorem \ref{thm1.1} 
is nothing but 
\cite[Theorem 1.1]{fujino-gongyo} (see also Remark \ref{rem1.2}). 
In this paper, we will show that we 
can reduce Theorem \ref{thm1.1} to 
\cite[Theorem 1.1]{fujino-gongyo}. 
We note that \cite[Theorem 1.1]{fujino-gongyo} 
plays a crucial role in the study of the abundance conjecture 
for semi-log canonical pairs. 
\end{rem}

\begin{rem}[Log canonical pairs of log general type]\label{rem1.4}
If $(X, \Delta)$ is a projective log canonical pair such that 
$K_X+\Delta$ is big, then $\Bir(X, \Delta)$ is already a finite group 
(see \cite[Corollary 3.13]{fujino-gongyo}). 
Therefore, Theorem \ref{thm1.1} is 
obvious when $K_X+\Delta$ is big. 
\end{rem}

\begin{rem}[Kawamata log terminal pairs]\label{rem1.5}
We have already established some better results for kawamata 
log terminal pairs. 
For the details, see \cite[\S 3.1]{fujino-gongyo}, 
where we do not need the minimal model program. 
Hence, in this paper, 
we are mainly interested in a log canonical pair $(X, \Delta)$ 
which is not kawamata log terminal. 
For smooth varieties, see also \cite[\S 14]{ueno}.   
\end{rem}

\begin{rem}\label{rem1.6} 
In a recent preprint \cite{fujino-finiteness}, 
the first author established the finiteness of relative log pluricanonical 
representations in a suitable complex analytic setting. 
For the details, see \cite{fujino-finiteness}. 
\end{rem}

Similarly to Theorem \ref{thm1.1}, we can prove the following theorem. 

\begin{thm}\label{thm1.7}
Let $V$ be a smooth quasi-projective variety and let $X$ be a smooth 
projective completion of $V$ such that 
$\Delta:=X\setminus V$ is a simple 
normal crossing divisor on $X$. 
Assume that $(X, \Delta)$ has a good minimal model. 
Then there exists a positive integer $k$ such that 
the image of 
\begin{equation*}
\rho_{km}\colon \PBir(V)\to \Aut_{\mathbb C} 
H^0(X, \mathcal O_X(km(K_X+\Delta)))
\end{equation*}
is a finite group for every positive integer $m$, 
where 
\begin{equation*}
\PBir(V):=\{\sigma\mid \text{$\sigma\colon V\dashrightarrow V$ is 
proper birational}\}. 
\end{equation*}
\end{thm}

Note that Theorem \ref{thm1.7} implies Theorem \ref{thm1.8}. 

\begin{thm}[Affine varieties]\label{thm1.8}
Let $V$ be a smooth affine variety and let $X$ be a smooth 
projective completion of $V$ such that 
$\Delta:=X\setminus V$ is a simple 
normal crossing divisor on $X$. 
Then there exists a positive integer $k$ such that 
the image of 
\begin{equation*}
\rho_{km}\colon \Aut(V)\to \Aut_{\mathbb C} 
H^0(X, \mathcal O_X(km(K_X+\Delta)))
\end{equation*}
is a finite group for every positive integer $m$, 
where 
\begin{equation*}
\Aut(V):=\{\sigma\mid \text{$\sigma\colon V\to V$ is an isomorphism}\}. 
\end{equation*} 
\end{thm}

To the best knowledge of the authors, Theorem \ref{thm1.8} is nontrivial and new. 

\begin{ack}\label{b-ack}
The first author was partially 
supported by JSPS KAKENHI Grant Numbers 
JP20H00111, JP21H04994, JP21H04994, JP23K20787. 
The authors would like to thank Professors Hideo Kojima and Takashi Kishimoto for 
comments. 
\end{ack}

Throughout this paper, we will work over $\mathbb C$, the field of 
complex numbers. We will freely use the standard notation as in 
\cite{fujino-fundamental} and \cite{fujino-foundations}. 

\section{Preliminaries}

In this section, we recall some basic definitions and 
properties necessary for this paper. For 
the standard notation of the 
theory of minimal models, see, for example, \cite{fujino-fundamental}, 
\cite{fujino-foundations}, and so on. Let us 
start with the definition of $B$-birational maps, which 
was first introduced in \cite{fujino-abundance}. 

\begin{defn}[$B$-birational maps]
Let $(X, \Delta)$ be a projective log canonical pair. 
We say that a birational map $\sigma\colon X\dashrightarrow 
X$ is {\em{$B$-birational}} if there exists a commutative diagram 
\begin{equation*}
\xymatrix{
&Z\ar[dl]_-\alpha\ar[dr]^-\beta& \\ 
X \ar@{-->}[rr]_-\sigma&&X
}
\end{equation*} 
such that $Z$ is a normal projective variety, $\alpha$ and 
$\beta$ are projective birational morphisms, and 
\begin{equation*}
\alpha^*(K_X+\Delta)=\beta^*(K_X+\Delta)
\end{equation*} 
holds.  
We put 
\begin{equation*}
\Bir(X, \Delta):=\{\sigma\mid 
\text{$\sigma\colon (X, \Delta)\dashrightarrow 
(X, \Delta)$ is $B$-birational}\}. 
\end{equation*}
Then $\Bir(X, \Delta)$ has a natural group structure. 
We take a positive integer $m$ such that $m(K_X+\Delta)$ is 
Cartier. 
Then it is easy to see that we have a group homomorphism 
\begin{equation*}
\rho_m\colon \Bir(X, \Delta)\to \Aut_{\mathbb C} H^0(X, \mathcal 
O_X(m(K_X+\Delta))). 
\end{equation*}
We call it the {\em{$B$-pluricanonical 
representation}} or {\em{log pluricanonical representation}} 
for the pair $(X, \Delta)$.  
\end{defn}

For the details of $\Bir(X, \Delta)$, see \cite{fujino-gongyo}. 
Let us recall the notion of proper birational maps 
(see, for example, \cite{iitaka}). 

\begin{defn}[Proper birational maps]\label{def2.2}
Let $V$ be a quasi-projective variety. 
We say that a birational map $\sigma\colon V\dashrightarrow 
V$ is {\em{proper birational}} if $p_1$ and $p_2$ are projective, where 
$\Gamma$ is the graph of $\sigma\colon V\dashrightarrow V$ and 
$p_1$ and $p_2$ are projections from $\Gamma$ to $V$ as in the following 
commutative diagram (see \cite[Proposition 2.17 (i)]{iitaka}). 
\begin{equation*}
\xymatrix{
&\Gamma\ar[dr]^-{p_2}\ar[dl]_-{p_1}& \\ 
V\ar@{-->}_-\sigma[rr]&&V
}
\end{equation*} 
We put  
\begin{equation*}
\PBir(V):=\{\sigma \mid\text{$\sigma\colon V\dashrightarrow 
V$ is proper birational}\}. 
\end{equation*} 
Then it is easy to see that $\PBir(V)$ has a natural group structure 
(see \cite[Proposition 2.17 (iii)]{iitaka}). 
We note that if $V$ is affine and normal then $\sigma$ is 
an isomorphism for every $\sigma\in \PBir(V)$ by 
Zariski's main theorem (see \cite[Theorem 2.19 and Corollary]{iitaka}). 
We further assume that $V$ is smooth and $X$ is a smooth 
projective completion of $V$ such that $\Delta:=X\setminus V$ 
is a simple normal crossing divisor on $X$. 
Let $\Gamma'\to \Gamma$ be a projective resolution of singularities and 
let $Y$ be a projective completion of $\Gamma'$ such that 
$\Delta_Y:=Y\setminus \Gamma'$ is a 
simple normal crossing divisor on $Y$. 
Then we have the following commutative diagram. 
\begin{equation*}
\xymatrix{
&\Gamma'\ar[d]\ar[ddr]^-{p'_2}\ar[ddl]_-{p'_1}& \\ 
&\Gamma\ar[dr]_-{p_2}\ar[dl]^-{p_1}& \\ 
V\ar@{-->}_-\sigma[rr]&&V
}
\end{equation*} 
Let $m$ be any positive integer. 
In this case, for every $\sigma\in \PBir(V)$, 
we can define $\sigma^* \in \Aut_{\mathbb C}H^0(X, 
\mathcal O_X(m(K_X+\Delta)))$ as follows:  
\begin{equation*}
\begin{split}
\sigma^*:=(\overline {p}'^*_1)^{-1}\circ 
\overline{p}'^*_2\colon H^0(X, 
\mathcal O_X(m(K_X+\Delta)))&\to 
H^0(Y, 
\mathcal O_Y(m(K_Y+\Delta_Y)))\\&\to 
H^0(X, 
\mathcal O_X(m(K_X+\Delta))),  
\end{split} 
\end{equation*}
where $\overline {p}'_i\colon Y\to X$ is the extension of $p'_i\colon 
\Gamma' \to V$ for $i=1, 2$.  
We can easily see that it is independent of the choice of $(Y, \Delta_Y)$ 
(see \cite[Theorem 11.1]{iitaka}). 
Then we can consider the following group homomorphism 
\begin{equation*}
\rho_m\colon \PBir(V)\to \Aut_{\mathbb C} H^0(X, 
\mathcal O_X(m(K_X+\Delta)))
\end{equation*} 
by setting $\rho_m(\sigma):=\sigma^*$ for every positive integer $m$. 
\end{defn}

For the sake of completeness, we recall the 
definition of good minimal models in the sense of Birkar--Shokurov. 

\begin{defn}[Good minimal models]\label{def2.3}
Let $(X, \Delta)$ be a projective log canonical pair. 
Let $\psi\colon X\dashrightarrow X'$ be a birational map 
and let $E$ be the reduced $\psi^{-1}$-exceptional divisor on $X'$, 
that is, $E=\sum _jE_j$ where 
$E_j$ are the $\psi^{-1}$-exceptional prime divisors on $X'$. 
Then $(X', \Delta')$ is called a {\em{good minimal model}} 
of $(X, \Delta)$ (in the sense of Birkar--Shokurov) 
if 
\begin{itemize}
\item[(1)] $(X', \Delta')$ is a projective $\mathbb Q$-factorial 
divisorial log terminal pair, where $\Delta':=\psi_*\Delta+E$, 
\item[(2)] $K_{X'}+\Delta'$ is semiample, and 
\item[(3)] $a(D, X, \Delta)<a(D, X', \Delta')$ holds for 
every $\psi$-exceptional prime divisor $D$ on $X$. 
\end{itemize}  
\end{defn}

In Definition \ref{def2.3}, we note that we can 
prove $a(P, X, \Delta)\leq a(P, X', \Delta')$ for every 
prime divisor $P$ over $X$ by the negativity lemma. 

\section{Proofs}

In this section, we prove Theorems \ref{thm1.1}, 
\ref{thm1.7}, and \ref{thm1.8}. 

\begin{proof}[Proof of Theorem \ref{thm1.1}]
Let $\psi\colon (X, \Delta)\dashrightarrow (X', \Delta')$ be 
a good minimal model. Then we can construct the following commutative 
diagram
\begin{equation*}
\xymatrix{
&&\ar[dll]_-{\alpha'}\ar[dl]^-\alpha(Y, \Delta_Y)\ar[dr]_-\beta
\ar[drr]^{\beta'}&& \\ 
(X', \Delta') & \ar@{-->}[l]^-\psi(X, \Delta) \ar@{-->}[rr]_-\sigma&& (X, \Delta)
\ar@{-->}[r]_\psi& (X', \Delta')
}
\end{equation*}
where $Y$ is a smooth projective variety with 
\begin{equation*}
\alpha^*(K_X+\Delta)=:K_Y+\Delta_Y:=\beta^*(K_X+\Delta)
\end{equation*} 
such that the support of $\Delta_Y$ is a simple normal crossing 
divisor on $Y$. 
Then we can write 
\begin{equation*}
K_Y+\Delta^{>0}_Y=\alpha'^*(K_{X'}+\Delta')+E
\end{equation*} and 
\begin{equation*}
K_Y+\Delta^{>0}_Y=\beta'^*(K_{X'}+\Delta')+F
\end{equation*} 
such that $E$ and $F$ are both effective 
with $\alpha'_*E=0$ and $\beta'_*F=0$. 
We note that $E-F$ is $\alpha'$-nef and $\alpha'_*(F-E)\geq 0$. 
This implies $F\geq E$ by the well-known negativity lemma. Similarly, 
we can prove that $E\geq F$ holds. 
Thus we have $E=F$. 
This means that $\psi\circ \sigma \circ \psi^{-1}\in \Bir(X', \Delta')$. 
We take the smallest positive integer $k$ such that 
$k(K_X+\Delta)$ and $k(K_{X'}+\Delta')$ are both Cartier. 
By \cite[Theorem 1.1]{fujino-gongyo}, 
we know that 
the image of 
\begin{equation*}
\rho'_{km}\colon \Bir(X', \Delta')\to \Aut _{\mathbb C} H^0(X', 
\mathcal O_{X'}(km(K_{X'}+\Delta'))) 
\end{equation*} 
is a finite group for every positive integer $m$. 
We have the following commutative diagram: 
\begin{equation*}
\xymatrix{ 
\Bir(X, \Delta)\ar@{^{(}->}[d]_-{\pi}\ar[r]^-{\rho_{km}}& \Aut _{\mathbb C} H^0(X, 
\mathcal O_X(km(K_X+\Delta)))\ar[d]^{\simeq}_-{\Pi}\\
\Bir(X', \Delta')\ar[r]_-{\rho'_{km}}&\Aut _{\mathbb C} H^0(X', 
\mathcal O_{X'}(km(K_{X'}+\Delta'))), 
}
\end{equation*}
where $\pi(\sigma)=\psi\circ \sigma\circ \psi^{-1}$ for 
$\sigma\in \Bir(X, \Delta)$ and $\Pi(g)=(\psi^{-1})^*\circ 
g\circ \psi^*$ for $g\in \Aut_{\mathbb C}H^0(X, \mathcal O_X(km(K_X+\Delta)))$. 
Hence we obtain the desired finiteness. 
We finish the proof of Theorem \ref{thm1.1}. 
\end{proof}

The proof of Theorem \ref{thm1.7} is essentially the 
same as that of Theorem \ref{thm1.1}. 

\begin{proof}[Proof of Theorem \ref{thm1.7}]
Let $\psi\colon (X, \Delta)\to (X', \Delta')$ be a good minimal model. 
As in Definition \ref{def2.2}, we construct $(Y, \Delta_Y)$ and 
obtain the following commutative diagram:   
\begin{equation*}
\xymatrix{
&&\ar[dll]_-{\alpha'}\ar[dl]^-\alpha(Y, \Delta_Y)\ar[dr]_-\beta
\ar[drr]^{\beta'}&& \\ 
(X', \Delta') & \ar@{-->}[l]^-\psi(X, \Delta) \ar@{-->}[rr]_-\sigma&& (X, \Delta)
\ar@{-->}[r]_\psi& (X', \Delta'), 
}
\end{equation*}
where $\alpha$ (resp.~$\beta$) is the extension of $p'_1\colon \Gamma'\to V$ 
(resp.~$p'_2\colon \Gamma'\to V$). 
By construction, we have 
$\Delta_Y=\alpha^{-1}(\Delta)=\beta^{-1}(\Delta)$ since 
$\sigma\colon V\dashrightarrow V$ is proper birational (see 
\cite[Lemma 11.2]{iitaka}). 
As in the proof of Theorem \ref{thm1.1}, we write 
\begin{equation*}
K_Y+\Delta_Y=\alpha'^*(K_{X'}+\Delta')+E
\end{equation*} 
and 
\begin{equation*}
K_Y+\Delta_Y=\beta'^*(K_{X'}+\Delta')+F 
\end{equation*} 
with $E\geq 0$, $F\geq 0$, $\alpha'_*E=0$, and $\beta'_*F=0$. 
Then, by the negativity lemma, we can prove that 
$E=F$ holds, that is, $\psi\circ \sigma\circ \psi^{-1}\in 
\Bir(X', \Delta')$. 
We take the smallest positive integer $k$ such that 
$k(K_{X'}+\Delta')$ is Cartier. Then, by \cite[Theorem 1.1]{fujino-gongyo}. 
the image of 
\begin{equation*}
\rho'_{km}\colon \Bir(X', \Delta')\to 
\Aut_{\mathbb C} H^0(X', \mathcal O_{X'}(km(K_{X'}+\Delta')))
\end{equation*} 
is a finite group for every positive integer $m$. 
Hence, by the same argument as in the proof 
of Theorem \ref{thm1.1}, we obtain the desired finiteness. We finish 
the proof of Theorem \ref{thm1.7}. 
\end{proof}

Finally, we prove Theorem \ref{thm1.8}, which is an easy application of Theorem 
\ref{thm1.7}. 

\begin{proof}[Proof of Theorem \ref{thm1.8}]
As we mentioned before, $\PBir(V)=\Aut(V)$ holds by 
Zariski's main theorem 
since $V$ is a smooth affine variety. 
Moreover, it is well known that $(X, \Delta)$ has a good minimal 
model when $\kappa (X, K_X+\Delta)\geq 0$ since $V$ is 
affine (see the proof of 
\cite[Proposition 4.1]{fujino-sub}). 
Thus, the desired statement follows from Theorem \ref{thm1.7}. 
We finish the proof of Theorem \ref{thm1.8}. 
\end{proof}



\begin{thebibliography}{HX} 

\bibitem[F1]{fujino-abundance} 
O.~Fujino, Abundance theorem for semi log canonical 
threefolds, Duke Math. J. \textbf{102} (2000), no. 3, 513--532. 

\bibitem[F2]{fujino-fundamental} 
O.~Fujino, 
Fundamental theorems for the log minimal model program, 
Publ. Res. Inst. Math. Sci. \textbf{47} (2011), no. 3, 727--789.

\bibitem[F3]{fujino-foundations}
O.~Fujino, Foundations of the minimal model 
program, MSJ Memoirs, \textbf{35}. Mathematical 
Society of Japan, Tokyo, 2017. 

\bibitem[F4]{fujino-sub} 
O.~Fujino, 
On subadditivity of the logarithmic Kodaira dimension, 
J. Math. Soc. Japan \textbf{69} (2017), no. 4, 1565--1581. 

\bibitem[F5]{fujino-finiteness} 
O.~Fujino, On finiteness of relative log pluricanonical 
representations, preprint (2024). 

\bibitem[FG]{fujino-gongyo}
O.~Fujino, Y.~Gongyo, 
Log pluricanonical representations and 
the abundance conjecture, 
Compos. Math. \textbf{150} (2014), no. 4, 593--620.

\bibitem[HX]{hacon-xu}
C.~D.~Hacon, C.~Xu, 
On finiteness of $B$-representations and semi-log canonical abundance, 
{\em{Minimal models and extremal rays (Kyoto, 2011)}}, 361--377,
Adv. Stud. Pure Math., \textbf{70}, Math. Soc. Japan, [Tokyo], 2016.

\bibitem[I]{iitaka} 
S.~Iitaka, {\em{Algebraic geometry. An introduction to birational geometry of 
algebraic varieties}}, North-Holland Mathematical Library, \textbf{24}. Graduate 
Texts in Mathematics, \textbf{76}. Springer-Verlag, New York-Berlin, 1982.

\bibitem[U]{ueno} 
K.~Ueno, {\em{Classification theory of algebraic 
varieties and compact complex spaces}}, 
Notes written in collaboration with P.~Cherenack, 
Lecture Notes in Mathematics, Vol. \textbf{439}. Springer-Verlag, 
Berlin-New York, 1975.

\end{thebibliography}
\end{document}